\documentclass[letterpaper, 10 pt, conference]{ieeeconf}  

\IEEEoverridecommandlockouts                              
                                                          
\usepackage{color}
\usepackage{multirow}
\usepackage{graphicx}
\usepackage{subfigure}
\usepackage{amsmath}
\usepackage{amsfonts}
\usepackage{cite}
\usepackage{soul}

\usepackage{algorithm}
\usepackage[noend]{algpseudocode}

\newtheorem{lemma}{Lemma}
\newtheorem{theorem}{Theorem}
\newtheorem{definition}{Definition}[section]

\newcommand{\cK}{{\cal K}}

\newcommand{\cG}{{\cal G}}

\newcommand{\cD}{{\cal D}}
\newcommand{\cO}{{\cal O}}
\newcommand{\cR}{{\cal R}}

\newcommand{\diag}{\mathop{\bf diag}}
\newtheorem{DE}{Definition}[section]
\newtheorem{NO}[DE]{Notation}

\newtheorem{RE}[DE]{Remark}

\newtheorem{TEMP}{Template}

\newcommand{\argmax}{\mathop{\rm argmax}}

\title{\LARGE \bf
Robust Control of Cascading Power Grid Failures using Stochastic Approximation
}

\author{Daniel~Bienstock$^{1}$ and Guy Grebla$^{2}$
\thanks{*This research
was partially funded by LANL award ``Grid Science'' and DTRA award  HDTRA1-13-1-0021.}
\thanks{$^{1}$D. Bienstock is with the Departments 
of Industrial Engineering and Operations Research, and Applied Physics and Applied Mathematics, Columbia University, New York, NY, 10027 USA.
{\tt\small email: dano@columbia.edu} }
\thanks{$^{2}$G. Grebla is with the Department of Electrical Engineering, Columbia University, New York, NY, 10027 USA.
{\tt\small email: guy@ee.columbia.edu} }
}

\begin{document}

\maketitle

\begin{abstract}
Cascading failure of a power transmission system are initiated by an exogenous event that disable a 
set of elements (e.g., lines) followed by a sequence of interrelated failures (or more precisely,
trips) of overloaded elements caused by the combination of physics of power flows in the changed
system topology, and controls.  Should this sequence accelerate it can lead to a large system failure
with significant loss of load.  In previous work we have analyzed deterministic algorithms that in
an online fashion (i.e., responding to observed data) selectively shed load so as to minimize the 
amount of lost load at termination of the cascade.  In this work we present a rigorous methodology for
incorporating noise and model errors, based on the Sample Average Approximation methodology for 
stochastic optimization.
\end{abstract}
\section{Introduction}
We present a rigorous methodology for computing robust algorithms for control of cascading
failures of power transmission systems.   We focus on the linearized, or DC, approximation for power 
flows, and on line tripping outages.  We note, however, that the same underlying methodology will 
apply for other models of power flows and other types of equipment outages.
This work extends the approach in \cite{bienstock2011optimal}; related
work is consided in \cite{chertkovpes11}, \cite{carreras2}, \cite{pooyaeppsteinhines}, and references therein.

 In the DC approximation to power flows (see \cite{bergenvittal} for background) a transmission system is modeled by
given a directed graph $G$ with $n$ buses and $m$ lines.  In addition, for each line $e$ we are given 
its {\em reactance} $x_e$ and its {\em limit} $u_e$ (we may also refer to a line $e$ in the form
$e = pq$ so as to indicate its ``from'' bus $p$ and `its `to'' buse $q$.  Additionally, we are given a {\em supply-demand} vector $\beta \in \cR^n$ with the following interpretation.  For a bus $i$, if $\beta_i > 0$ then $i$ is a {\em generator} (a source node) while if
$\beta_i < 0$ then $i$ is a {\em load} (a demand node) and in that case 
$-\beta_i$ is the {\em demand} $d_i$ at $i$.  The condition $\sum_i \beta_i = 0$ is
assumed to hold.  We denote by $\cG$ the set of generators and by $\cD$ the
set of demand nodes.  The linearized power flow problem specifies a variable $f_{pq}$ associated 
with each line $pq$ (active power flow) and a variable $\phi_p$ (phase angle) associated with each bus 
$p$.  The DC approximation is given by the system
 system of equations:  
\begin{eqnarray}
&& N f \ = \ \beta, \ \ \ N^T \phi - X f \ = \ 0, \label{linearpowerflow}\\[-15pt]\nonumber
\end{eqnarray}
\noindent where $N$ denotes the node-arc incidence matrix of $G$ \cite{AMO93} and 
$X = \diag\{x_{ij}\}$.   
\begin{RE}\label{uniqueness}
It can easily be shown that system (\ref{linearpowerflow}) is feasible if and only if $\sum_{i \in K} \beta_i = 0$ for each
component (``island'') $K$ of $G$, and in that case the solution is unique
in the $f$ variables.  
\end{RE}

In this statement the line limits $u_{pq}$ play no role.  It may be the case that in the (unique)
solution $f$ to (\ref{linearpowerflow}) we have $|f_{pq}| > u_{pq}$.  In that case line $pq$ is
\textit{at risk}, and, unless some control action is taken, will eventually trip.  The ``eventually'' is
an imprecise statement.  Certainly, in the context of cascade modeling, the $u_{pq}$ should be the 
``emergency'' line limits.  Nevertheless, care should be taken to model the non-precise nature of
line tripping.  We will comment on this point, again, later. 

We can now outline the model for discrete-time cascading failures   
from \cite{bienstock2011optimal}.  Similar models have been proposed in the literature, see

\begin{center}
  \fbox{
    \begin{minipage}{0.9\linewidth}
      \hspace*{.9in} \begin{TEMP}{GENERIC CASCADE TEMPLATE}\label{gencasc} \end{TEMP}
      {\bf Input}: a power grid with graph $G$ (post-initiating event). Set $G^1 = G$.\\
      {\bf For} $t = 1, 2, \ldots$ {\bf Do} \\
        \hspace*{.2in}(comment: time-step $t$ of the cascade)\\
        \hspace*{.2in}{\bf 1.} Set $f^t = $ vector of power flows in $G^t$.\\
        \hspace*{.2in}{\bf 2.} Set $\cO^{t} = $ set of lines of $G^t$ that become outaged \\
        \hspace*{.4in}in time-step $t$. \\
        \hspace*{.2in}{\bf 3.} Set $G^{t+1} = G^{t} - \cO^t$.  Adjust loads and\\ 
        \hspace*{.4in}generation in $G^{t}$.
    \end{minipage}
  }
\end{center}
In this template, each ``time-step'' models a time increment of length $\Delta t > 0$.  In Step 3, when a load imbalance results in some island, we assume that
the larger of generation or load are adjusted, downwards, so as to
attain balance. The model does
not attempt to explain the cascade within resolution finer than $\Delta t$. It should be noted that 
cascades can be extremely rapid -- perhaps lasting seconds.  
However, several cascading failures of notoriety have been significantly slower, such as the 2004 U.S. 
Northeast event \cite{usc}.  For the purposes of this paper, the reader should assume that $\Delta t$ is
of the order of minutes.  We remark that cascades are extremely noisy events, characterized at a fine
time scale by a myriad of events, some of which are extremely challenging to model, such as 
physical contacts of sagging lines with vegetation, and human action (and errors -- which have been
found to take place in many cascades).   Any model of cascading failures which attempts to model time
at a very fine scale (perhaps, continuous time) would need incorporate correspondingly fine models for
such events.  Additionally, there would be a need to resolve ``race-conditions,''  which take place when
many lines can trip within a very short amount of time with potentially very different cascade outcomes (observed by other authors).  
In this paper, for simplicity, we will therefore assume a relatively large $\Delta t$.  We note
that one could
furthermore argue that a control algorithm that is subject to noise and incomplete information would
benefit from relying on a relatively large $\Delta t$.

In \cite{bienstock2011optimal} we considered a general form of control
template, which assumes that the initial exogenous event that sets-off the
cascade has been observed, and that $G$ is the post-event network.  The control entails load shedding; when load
is shed in some connected subnetwork,
generation must be correspondingly adjusted in that subnetwork.

\algrenewcommand\algorithmicdo{}
\algrenewcommand\algorithmicfor{{\bf For}}
\algrenewcommand\algorithmicindent{1.5em}
\floatname{algorithm}{Framework}
\begin{algorithm}
\caption{Cascade Control} 
\label{contcasc}
{\bf Input}: a power grid with graph $G$. Set $G^{(1)} = G$.
\begin{algorithmic}[1]
\State{{\bf Compute} control algorithm}
\For{ $t = 1, 2, \ldots, T-1$}
\Comment{controlled time-step $t$ of the cascade}
\State{Obtain grid measurements} \label{ln:measurements}
\State{\bf Apply control} \label{ln:apply-control}
\State{Set $f^{(t)} = $ vector of resulting power flows in $G^{(t)}$} \label{ln:resulting-flows}
\State{{\parbox[t]{\dimexpr\linewidth-\algorithmicindent}{Set $\cO^{(t)} = $ set of lines of $G^{(t)}$ that become outaged in time-step
  $t$}}}
\State{Set $G^{(t+1)} = G^{(t)} - \cO^{(t)}$}
\State{Adjust loads and generation  in $G^{(t)}$} \label{ln:adjust}
\EndFor
\State{{\bf Termination} (time-step $T$). If $G^{(T)}$ has line
  overloads, proportionally shed demand until all line
  overloads are eliminated. Set $\psi^{(T)} \, \doteq \, \min\left\{ 1 \, , \,
    \max_{j} \{ \frac {|f^{(T)}_{j}|} { u_{j}} \} \right\}$. If $\psi^{(T)} > 1$, then
  any bus $v$ resets its demand to $d^{(T)}_v /\psi^{(T)}$. }
\end{algorithmic}
\end{algorithm}
We stress that $G$ is the {\bf post-initiation} state of the grid, i.e. the
elements disabled in the initial exogenous event are not present.
The termination condition ensures that indeed the cascade is stopped at the
end of the planning horizon \cite{bienstock2011optimal}.  In general, Step 3
consists of any step where a network controller obtains current data.  In Step 4
this data is used to apply the control computed in Step 1 using the measurements
as inputs.  
A completely deterministic version of Framework \ref{contcasc} is considered
in \cite{bienstock2011optimal}.  This algorithm is characterized by two
features:
\begin{itemize}
\item[(s.i)] At time $t$, if $K$ is an island in $G^{(t)}$, then in Step 4 all
  loads in $K$ are scaled by a factor $0 \le \lambda^t_K \le 1$ (as are all
  generator outputs in $K$).
\item [(s.ii)] In Step 6, a line $e$ is tripped if $u_e \le |g^{(t)}_e|$.
\end{itemize}
The objective is to compute the $\lambda^t$ so that at termination the total
load still being delivered is maximized.  The main contribution in \cite{bienstock2011optimal} is an algorithm that solves this problem in polynomial time,
for each fixed value of $T$.  This fact may appear surprising given the nature
of rule (s.i).  There are an exponential number of islands $K$ that could be realized
at time-step $t$ -- how could the algorithm run in polynomial time?

The answer is that the $\lambda^t$ parameters need only be computed for an
{\bf optimal} realization of the controlled cascade.  Having chosen the vectors
$\lambda^1, \lambda^2, \ldots, \lambda^{t-1}$ then a unique state will be
observed in time-step $t$. This is a consequence of the fact that rule (s.ii)
is purely deterministic.  Hence if we somehow know that $\lambda^1, \lambda^2, \ldots, \lambda^{t-1}$ have been optimally chosen, then all that is needed is
for $\lambda^{t}$ to be optimally chosen, as well.  Continuing inductively
we will obtain an optimal control across all time-steps.  To put it in a
pedestrian manner, the controlled cascade plays out like a script, with the
events that transpire at each time-step $t$ known, precisely, in advance.

Thus, the key in
the analysis of the algorithm in \cite{bienstock2011optimal} is that, indeed
we can choose $\lambda^{t}$ optimally (provided that $\lambda^1, \lambda^2, \ldots, \lambda^{t-1}$ have previously been computed optimally).  To prove this
point \cite{bienstock2011optimal} relies on a variant of dynamic programming.
Below we will consider an updated form of this algorithm.

\section{Modeling stochastics}
Even though the algorithm in \cite{bienstock2011optimal} is provably optimal,
it is readily apparent how that algorithm falls short, and in particular may
not prove robust.  This concerns rule (s.ii) -- should this assumption prove
inaccurate it is quite likely that the set of lines that trip in Step 6 will
be different than anticipated in the ``script'' mentioned above.  Not only that,
but the set of islands actually observed at time-step $t$ may be different from
those expected by the ``script'', and thus the computed control does not even
make sense (i.e. we have the wrong parameters $\lambda^t_K$).  In this section
we consider algorithms that not only bypass these shortcomings, but also attain
a form of algorithmic robustness that can be precisely stated.  
\begin{itemize}
\item [(L.a)] We will incorporate stochastics into the line-tripping rule. More
  precisely, at time $t$ we compute, for  each line $i$, the quantity 
$\tilde{f}^t_i=(1+\epsilon^t_i) |f^{(t)}_i|$, where 
where $\epsilon^t_i$ is a random variable from a known distribution. 
Additionally, we assume that $|\epsilon^t_i|$  is bounded by $1 > b\geq 0$.  Here, it is assumed that $b << 1$.  The
\textit{stochastic line-tripping rule} is that line $i$ is tripped if $\tilde f^t_i \ge u_i$.  In what follows, we will $\tilde{f}^t_i$ is the \textit{noisy flow} in line $i$ at time-step $t$.
\item [(L.b)] Under the stochastic line tripping rule, control rule (s.i) does not
  make sense because we do not know, in advance, the set of islands that will be
  realized at time-step $t$. In fact, potentially, any island could be realized.
  Instead, the control will compute a \textit{single} value $\lambda^t$ which
  will be used to scale all loads at time-step $t$.
\end{itemize}
We can now state the optimization problem of interest.
\begin{definition} {\bf Optimal Robust Control Problem}.  Compute scaling
values $\lambda^1, \ldots, \lambda^{T-1}$ such that subject to rules (L.a), (L.b) we maximize the \textit{expected yield},
where by ``yield'' we mean the load that is served at termination.  
\end{definition}

{\bf Remarks.} 
\begin{itemize}
\item In (L.a) we do not make any assumptions as to the source
or nature of the stochastics.  In fact, we generically model ``noise'' in this
fashion, so as to be able to capture \textit{any} form of uncertainty that could
hamper a control algorithm, so long as the magnitude of the errors is not
overly large.  As we will see in our experiments, simply allowing stochastic
tripping gives rise to a large variety of cascading outcomes.  A control that
maximizes expected load will thus be robust with respect to many alternative
histories that the system could follow. However, there is a specific setting
in which rule makes sense -- we can use (L.a) to measure \textit{errors} in 
line measurements.  
\item The termination rule in Framework \ref{contcasc} needs to be revised when
rule (L.a) is applied so as to make certain that the final load shedding does
terminate the cascade.  We do so by redefining 
$\psi^{(T)} \, \doteq \, \min\left\{ 1 \, , \,
    \max_{j} \{ \frac {\tilde f^{(T)}_{j}} { (1-b)u_{j}} \} \right\}$
\end{itemize}
Since $b$ is small, this amounts to a small correction, at termination.
\section{Solving the optimal robust control problem under a given noise vector
$\epsilon$}
Some of the algorithmic steps presented here echo some steps in \cite{bienstock2011optimal}.
However, because of (L.a) and (L.b), the underlying mathematical nature 
of the problem is fundamentally different, as are the actual proofs. First we introduce some notation. 
\begin{itemize} 
\item We define the function $\eta^{(T)}_G(z | \tilde{f})$
as the maximum expected yield given that in the first time-step the noisy flows
are $\tilde{f}^{(1)} = z \tilde{f}$.
\item We denote $G^t$ the power grid at the beginning of time-step $t$, a random 
variable under (L.a).  
\item Denote by $\epsilon$ the $m\times T$ matrix of $\epsilon_i^{(t)}$ values for all
$i$ and $1\leq t\leq m$. 
\item For $z \ge 0$ real, let $\Theta^{(T)}_G(z | f, \epsilon)$ be the
{\em deterministic} maximum yield obtained if at the start of the first time-step
the flows in the grid are equal $z f$, and all the $\epsilon$ is a given,
fixed, vector of values $\epsilon_i^{(t)}$ (rather than random). 
Note that $\Theta^{(T)}_G$ is deterministic since all
realizations of the random variables $\epsilon_i^{(t)}$ are provided. 
\end{itemize}
Clearly the
following holds,
\begin{eqnarray}
  \eta^{(T)}_G(z |
  \tilde{f}) = E_{\epsilon} \left [ \Theta^{(T)}_G(z | f',
\epsilon) \right ],
\end{eqnarray}
where $f'_i = \frac {\tilde{f}_i} {1+\epsilon_i^{(T)}}$.

\begin{lemma} \label{lem:breakpoints2}
  $\Theta^{(1)}_G(z | f', \epsilon)$ is a nondecreasing piecewise-linear
  function of $z$ with two pieces, the second one of which has zero slope.
\end{lemma}

\begin{proof}
  Note that since $T = 1$, only termination step in Framework \ref{contcasc}
  will be executed.  Denoting by $\tilde D$ the sum of demands implied by $f'$
  we have as per our cascade termination criterion that the final total demand
  at the end of $T = 1$ time-step will equal
  \begin{eqnarray}
    z \tilde D, && \mbox{if} \ \  z \, \le \, 1/\psi^{(1)}, \ \ \ \mbox{and} \\
    \frac{z}{z \psi^{(1)}} ~ \tilde D \ = \frac{1}{\psi^{(1)}}~ \tilde D, && \mbox{otherwise}.\ \ \hspace*{.2in} 
  \end{eqnarray}
\end{proof}

Using similar ideas to those shown in \cite{bienstock2011optimal}, we now turn
to the general case with $T>1$ and we will show the following theorem.
\begin{theorem} \label{th:breakpoints}
  $\Theta^{(T)}_G(z | f, \epsilon)$ has at most $\frac {m!} {(m-T+1)!}$ breakpoints.
\end{theorem}

To prove Theorem \ref{th:breakpoints}, we will need the following definitions.
\begin{definition}\label{def:crit}
  A {\em critical point} is a real $\gamma > 0$, such that for
some line $j$, $\gamma \tilde f_{j} = u_{j}$, i.e.
$\gamma |f_{j}| = u_{j}/(1 + \epsilon^t_j)$.
\end{definition}
Remark: for completeness we should use, in this definition, the superindex $t$,
which we have skipped for brevity.
Recall that we assume $u_j > 0$ for all $j$; thus let $0 < \gamma_1 < \gamma_2 < \ldots < \gamma_p$ be the set of all distinct
critical points.  Here $0 \le p \le m$.  Write $\gamma_0 = 0$ and 
$\gamma_{p+1} = +\infty$. 

\begin{definition}
For $1 \le i \le p$ let $F^{(i)} = \{ j: \, \gamma_h \tilde f_{j}  = u_{j} \}.$ 
\end{definition}

Assume that the initial flow is $z f$ with $ 1 \geq z > 0$ and
let $0 < \lambda^{(1)} \le 1$ be the optimal multiplier used to scale demands in
time-step 1.  Write
\begin{eqnarray}
 && q(z) = \argmax\{ h \, : \, \gamma_h < z \}. \label{defq}
\end{eqnarray}
Thus, $z \le \gamma_{q+1}$, and so $\lambda^{(1)} z \le \gamma_{q+1}$. We stress that 
these relationships remain valid in the boundary cases $q = 0$ and $q = p$.

\begin{NO} Let the index $i$ be such that $\lambda^{(1)} z \in (\gamma_{i-1}, \gamma_{i}]$. \end{NO}

In time-step 1 of Framework \ref{contcasc}, at line \ref{ln:apply-control} we will
scale all demands by $\lambda^{(1)}$. We assume the framework is given a connected
graph, namely, $G^{(1)}$ is connected. Therefore, in line \ref{ln:resulting-flows}
we will also scale all supplies by $\lambda^{(1)}$. Thus, for any $h \le i-1$, and
any line $j \in F^{(h)}$, we have that after Step \ref{ln:resulting-flows} the
absolute value of $f_j$ is $\lambda^{(1)} z \tilde f_{j}  > \gamma_h \tilde f_{j}  =
u_{j}$, and consequently line $j$ becomes outaged in time-step 1.  On the other
hand, for any line $j \notin \cup_{h \le i-1} F^{(h)}$, the absolute value of the
flow on $j$ immediately after Step \ref{ln:resulting-flows} is $\lambda^{(1)} z \tilde f_{j}  \le \gamma_{i} \tilde f_{j}  \le u_{j}$, and so line $j$ does not become
outaged in step 1.  In summary, the set of outaged lines is $\cup_{h \le i-1}
F^{(h)}$; in other words, we obtain the same network $G^{(2)} = G^{(1)} \setminus \cup_{h
  = 1}^{i-1} F^{(h)}$ for every $z$ with $\lambda^{(1)} z \in (\gamma_{i-1},
\gamma_{i}]$.

\begin{NO} For an index $j$, write $\cK(j)$ = set of components of $G^{(1)}
  \setminus \cup_{h = 1}^{j} F^{(h)}$. \end{NO}

Let $H \in \cK(i-1)$ and denote the initial supply-demand vector by $\beta$.
Prior to line \ref{ln:adjust} in step 1, the supply-demand vector for $H$ is
precisely the restriction of $\lambda^{(1)} z \beta$ to the buses of $H$, and when
we adjust supplies and demands in line \ref{ln:adjust}, we will proceed as
follows
\begin{itemize}
\item if $ \sum_{s \in \cD \cap H} (-\lambda^{(1)} z \beta_s) \, \ge \, \sum_{s \in \cG \cap H} (\lambda^{(1)}  z \beta_s)$ then for each demand bus $s \in \cD \cap H$ we will reset
its demand to $ -r \lambda^{(1)}  z \beta_s$, where
$$ r = \frac{\sum_{s \in \cG \cap H} (\lambda^{(1)} z \beta_s)}{\sum_{s \in \cD \cap H} (-\lambda^{(1)}  z \beta_s)} = -\frac{\sum_{s \in \cG \cap H} (\beta_s)}{\sum_{s \in \cD \cap H} (\beta_s)},$$
and we will leave all supplies in $H$ unchanged. 
\item likewise, if
$ \sum_{s \in \cD \cap H} (-z \lambda^{(1)} \beta_s) \, < \, \sum_{s \in \cG \cap H} (\lambda^{(1)}  z \beta_s)$ then the supply at each bus $s \in \cG \cap H$ will be reset
to $ r \lambda^{(1)} z \beta_s$, where
$$r = -\frac{\sum_{s \in \cD \cap H} (\beta_s)}{\sum_{s \in \cG \cap H}} (\beta_s),$$
but we will leave all demands in $H$ unchanged.
\end{itemize}
Note that in either case, in time-step 2 component $H$ will have a supply-demand
vector of the form $\lambda^{(1)} z \beta^H$, where $\beta^H$ is a supply-demand
vector which is {\em independent} of $z$. The supply-demand vector $\beta^H$
corresponds with flows $f^H$ on the lines in $H$, which are therefore also
{\em independent} of $z$. The flows $f^{(2)}$ on $G^{(2)}$ are $\cup_{H\in
  G^{(2)}} f^H$, and the final total demand will be
\begin{eqnarray}
\Theta_{G^{(2)}}^{(T-1)}(\lambda^{(1)} z | f^{(2)}) \label{Hfinal}
\end{eqnarray}

For a given value of $i$, since $z\leq 1$, also $i \leq q$ holds.  As noted
above, by definition (\ref{defq}) of $q$ we have that $\gamma_{i} \le
\gamma_{q} < z$.  Thus, the expression in (\ref{Hfinal}) is maximized when
$\lambda^{(1)} = \frac{\gamma_{i}}{z}$, and we obtain final ($T$-step) demand
equal to
\begin{eqnarray}
D_i & \doteq & \Theta_{G^{(2)}}^{(T-1)}(\gamma_{i} | \cup_{H\in
  \cK(i-1)} f^H), \label{fixedcase1}
\end{eqnarray}

In summary,
\begin{eqnarray}
\Theta_{G^{(1)}}^{(T)}(z | f) & = & \max _{1 \leq i\leq q(z)} D_i \label{eq:final}
\end{eqnarray}
We are now ready to prove Theorem \ref{th:breakpoints}

\begin{proof}
  In Lemma \ref{lem:breakpoints2}, we showed that for $T=1$ the function has
  two breakpoints.  As can be seen from \eqref{eq:final}, the number of
  breakpoints is at most $q(z)$. It is important to note that (a) $q(z)$ is at
  most the number of lines in the graph and (b) the number of non-tripped lines decreases
  by at least $1$ in every time-step (otherwise, the cascade stopped). Therefore,
  the number of breakpoints for $T$ steps is $m (m-1) \ldots (m-T+1) = \frac
  {m!} {(m-T)!}$.
\end{proof}

\section{Solving the optimal robust control problem using the Sample Average Approximation Method}
The key methodology that we will reply in order to solve the problem of interest
is the Sample Average Approximation (SAA) method~\cite{KS99}. The method
begins by taking $m T N$ i.i.d
samples (each under the assumed distribution) of $\epsilon$ values.
We view these values as arranged into $N$ ensembles, or \textit{realizations}. Here, for $1 \le k \le N$
we let $\epsilon^k$ denote the $k^{th}$ realization, consisting of values 
$\epsilon^{(t),k}_i$, for $1 \le t \le T$ and all lines $i$.  

\begin{definition}
Given any candidate vector $\Lambda = (\lambda^{(1)}, \lambda^{(2)}, \ldots, \lambda^{(T-1)})$ for solving the robust optimal control problem, denote
by $\Theta_G^T(z | f, \Lambda, k)$ to be the
yield obtained at termination if at the start of the first time-step
the flows in the grid are equal $z f$, if we use $\Lambda$ as the control
vector at each time-step, and $\epsilon^{(t),k}_i$ is used for each $t$ and $i$
in the line-tripping rule (L.a).
\end{definition}

Thus, $\Theta_G^T(1 | f^{(1)}, \Lambda, k)$ indicates the behavior of the 
control vector $\Lambda$ under the fixed noise vector $\epsilon^{k}$.  This 
motivates the following definition.
\begin{definition} {\bf (Sample average robust control problem.)}  Given the
$N$ realizations $\epsilon^{1}, \ldots, \epsilon^{N}$, compute a control 
vector $\Lambda$ so as to maximize
\begin{eqnarray}
&& \frac{1}{N} \sum_{k = 1}^N \Theta_G^T(1 | f^{(1)}, \Lambda, k). \label{saasum}\end{eqnarray}
\end{definition}
This definition is appealing in that, clearly, if $N$ is large a control that
maximimized (\ref{saasum}) should clearly perform well in our stochastic setting -- since it maximizes the average outcome over many possible noise scenarios.  
This observation can in fact be rigorously established by using the Central Limit Theorem (and observing that all random variables under consideration are bounded). See e.g. equation (2.23)
from~\cite{KS99}. If we want to solve the Robust Optimal Control problem within
additive error $\delta > 0$ with probability at least $1 - \alpha$
one can establish an upper bound of the form $N = O( \frac {3\sigma^2} {\delta^2/4} \log(1/\alpha)$, where $\sigma$ is the standard deviation of the distribution of 
the $\epsilon$ values \footnote{Some technical points skipped for brevity.} We
can establish an analogue of Lemma \ref{lem:breakpoints2}.  Let 
$\hat \eta^T_G (z| f) \ \doteq \ \frac{1}{N} \sum_{k = 1}^N \Theta_G^T(z | f, \Lambda, k)$.  This is the maximum average yield at termination under all the
realizations if we start time-step 1 with flows $z f$.

\begin{lemma} \label{lem:breakpoints}
$\hat \eta^T_G  (z| f)$ is nondecreasing, piecewise-linear, with at most
\begin{equation*}
  N \frac {m!} {(m-T)!}
\end{equation*}
breakpoints.
\end{lemma}

\begin{proof}
  By theorem \ref{th:breakpoints}, $\Theta^{(T)}_G(z | f, \epsilon)$ is
  piecewise-linear with at most $\frac {m!} {(m-T+1)!}$ breakpoints.  Since
  $\hat \eta^T_G  (z| f)$ is an average taken over $N$ instances of
  $\Theta^{(T)}_G(z | f, \epsilon)$, the lemma follows.
\end{proof}

Using Lemma~\ref{lem:breakpoints}, it follows that we can efficiently
determine the control in order to maximize the expected total demand in the
face of measurement errors.

\section{Simulation Study}

In this section we present simulation results using a modified form of the IEEE 118-bus system, with line limits set approximately at $20 \%$ above flow values\footnote{We will present experiments with larger examples at the conference}. We first describe our implementation and specific noise model (see (L.a) above). 
We then present and discuss the computed $\hat \eta^T_G  (z| f)$, obtained via the SAA, and compare the performance of our robust
control algorithm to its non-robust version from~\cite{bienstock2011optimal}.

\subsection{Implementation and Noise Model}

As discussed above an approximately optimal load shedding control can be found by
computing the function $\hat \eta^T_G  (z| f)$, and for this purpose
we will rely on Lemma~\ref{lem:breakpoints2}. Thus, at each every time-step
$t$, we solve the DC equations to obtain the flows and compute the critical
points $\gamma_1,\ldots,\gamma_p$. An important point is that these are the 
critical points arising from all realizations (recall Definition \ref{def:crit}).  Next, for each $i$ with $1 \le i \le p$ we proceed as follows. 

Note that any two choices for $\lambda^1$ in the open interval $(\gamma_{i-1}, \gamma_{i})$ will result in exactly the same set of line trips under all realizations.  In fact, any two choices for $\lambda^1$ in $(\gamma_{i-1}, \gamma_{i})$
will produce the same graph $G(i)$ (which we can easily compute) but with supply-demand
vectors that differ by a constant factor -- and thus, in a common flow vector
(up to scale) which we denote by $f(i)$.  It follows that if we (recursively)
compute the function $\hat \eta^{T-1}_{G(i)}  (z| f(i))$ (which is a function
of the real $z$) we will obtain $\hat \eta^{T}_{G}  (z| f)$ in the interval
$(\gamma_{i-1}, \gamma_{i})$.

Next we discuss the specific noise model used to implement rule (L.a).  We will
first describe a specific noise model.  Then we will prove it is of the 
form (L.a).  A line $j$ will trip
\begin{subequations}\label{practtrip}
\begin{eqnarray}
\hspace{-.5in}&& \mbox{with probability $1$, if $|f_j| \geq u_j$} \\
\hspace{-.5in}&& \mbox{with probability $1/2$, if $u_j > |f_j| \geq 0.95 u_j$}\\
\hspace{-.5in} && \mbox{with probability $0$, otherwise.}
\end{eqnarray}
\end{subequations}
\begin{lemma} Rule (\ref{practtrip}) is an example of (L.a).
\end{lemma} 
\noindent {\em Proof.}
Define the random variable $\epsilon_j^{(t)}$ for every $j$ and $t$ as
follows:
\begin{eqnarray}
\epsilon_j^{(t)} =  \frac{0.05}{0.95}, && \mbox{with probability } 0.5, \ \ \  \mbox{and} \\
\epsilon_j^{(t)} =  0, && \mbox{with probability 0.5}.\ \ \hspace*{.2in} 
\end{eqnarray}

Recall that $\tilde{f}_j^{(t)} = (1+\epsilon_j^{(t)}) f_j^t$ and consider the
case where $f_j^t \geq 0.95 u_j$. With probability $0.5$, $\epsilon_j^{(t)}=0$
and therefore $\tilde{f}j^{(t)}<u_j$ and line $j$ will not be tripped. But,
with probability $0.5$ we get $\epsilon_j^{(t)}=\frac {0.05}{0.95}$ and
$\tilde{f}j^{(t)} = (1+\frac{0.05} {0.95})f_j^t \geq (1+\frac{0.05}
{0.95})0.95 u_j = u_j$ and line $j$ will trip.  \QED

\subsection{Results}

We tested values of $T=2,3,4,5$. The initial cascade is caused by tripping line
$(4,5)$, since tripping this line causes a relatively major cascade in the
 118-bus system.

Fig.~\ref{fig:line2_realizations} shows the yield vs the scaling in the first
time step for a value of $T=3$. In Fig.~\ref{fig:line2_realizations}(a) the
number of realizations used for the Sample Average Approximation (SAA) is $5$
while in Fig.~\ref{fig:line2_realizations}(b) this number is increased to
$120$. As expected, increasing the number of realizations affects the expected
yield in the different linear segments. In Fig.~\ref{fig:line2_realizations},
we see that despite the fact that the graph differ, in this specific case the
maximum yield remains similar.

\begin{figure}
\centering
\subfigure[]{\includegraphics[width=0.485\columnwidth]{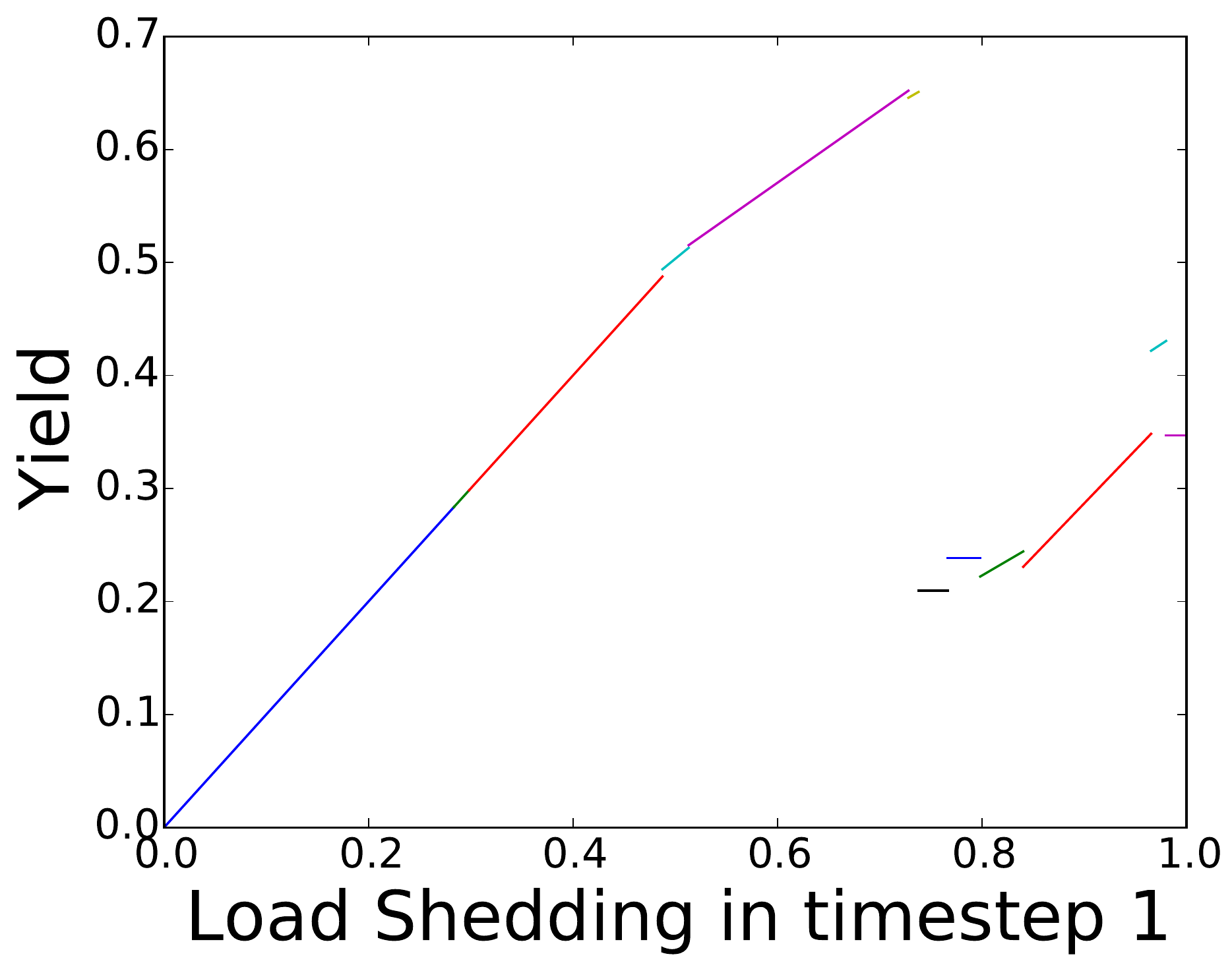}}
\subfigure[]{\includegraphics[width=0.485\columnwidth]{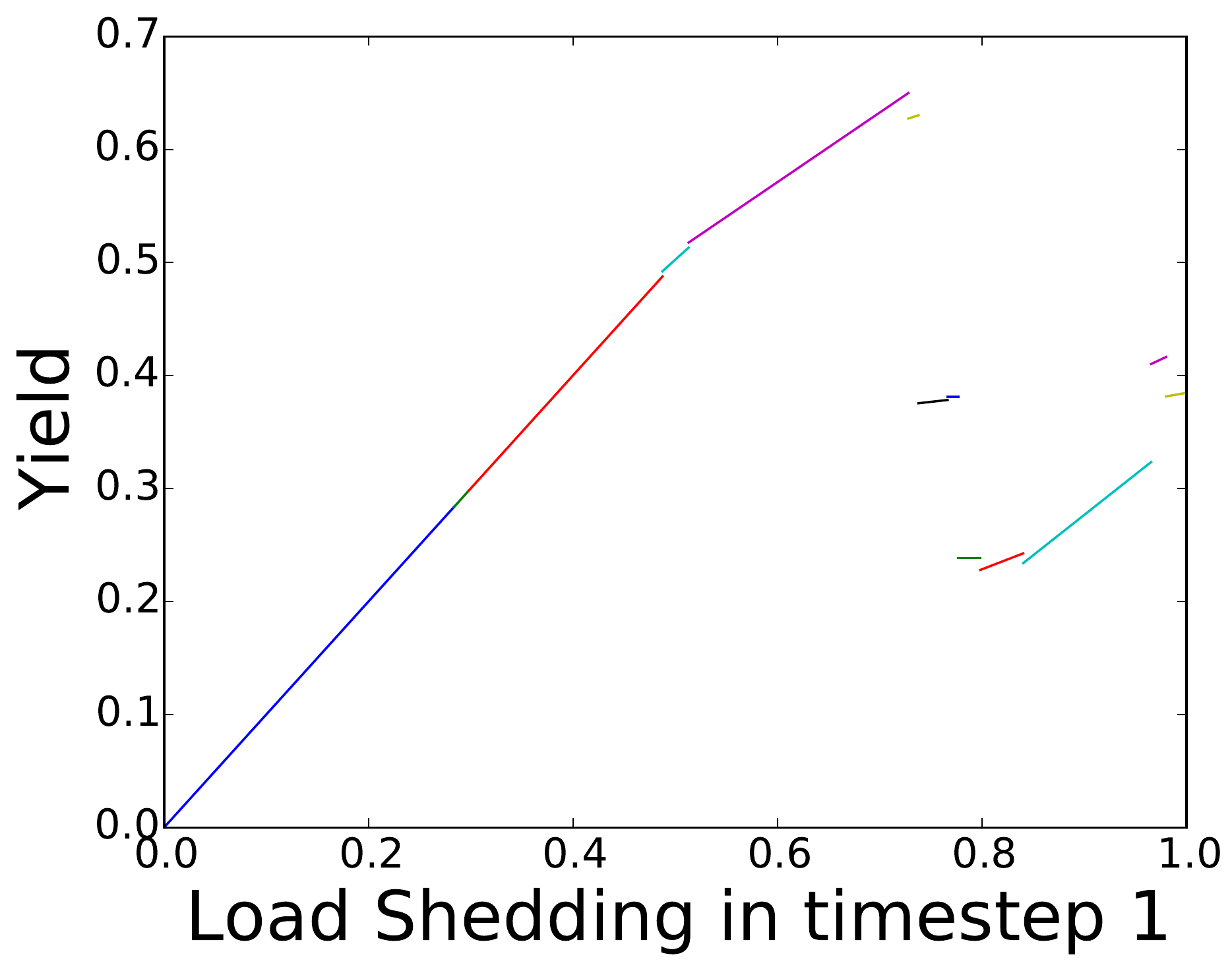}}
\caption{The yield $T=3$ in the first time step for: (a) 5 realizations, and
  (b) 120 realizations. \label{fig:line2_realizations}}
\end{figure}

While Fig.~\ref{fig:line2_realizations} shows the load shedding required in
the first timestep (shedding of $\approx 0.72$ maximize the yield), the figure
does not reveal what will be the load shedding at the second
time step. Fig.~\ref{fig:line2_2_vs_3}(a) shows the yield as a function of load
shedding in the second times tep {\em assuming that in the first time step an
  optimal choice for $\lambda^1$ is used}. As can be observed from
Fig.~\ref{fig:line2_2_vs_3}(a), the optimal choice for $\lambda^1$ is $\approx
0.85$. Fig.~\ref{fig:line2_2_vs_3}(b) considers the case where $T=2$, it can
be seen that that load shedding is done at only one time-step. While the maximum yield is
slightly lower, overall it is close to the maximum yield ($\approx 62\%$) for
the case where $T=3$. We will now show that the difference in the maximum
resulting from increasing $T$ can be significant.

\begin{figure}
\centering
\subfigure[]{\includegraphics[width=0.485\columnwidth]{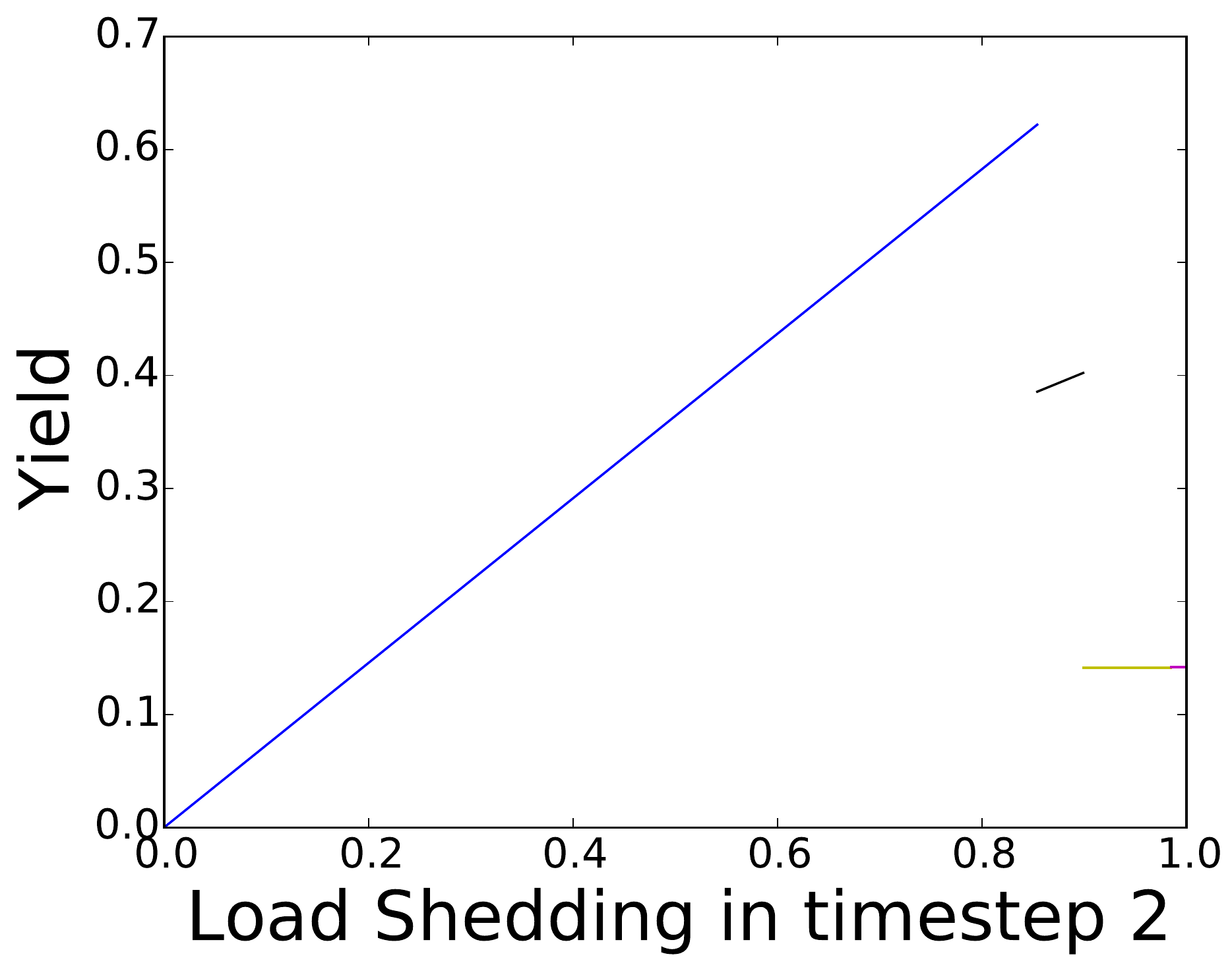}}
\subfigure[]{\includegraphics[width=0.485\columnwidth]{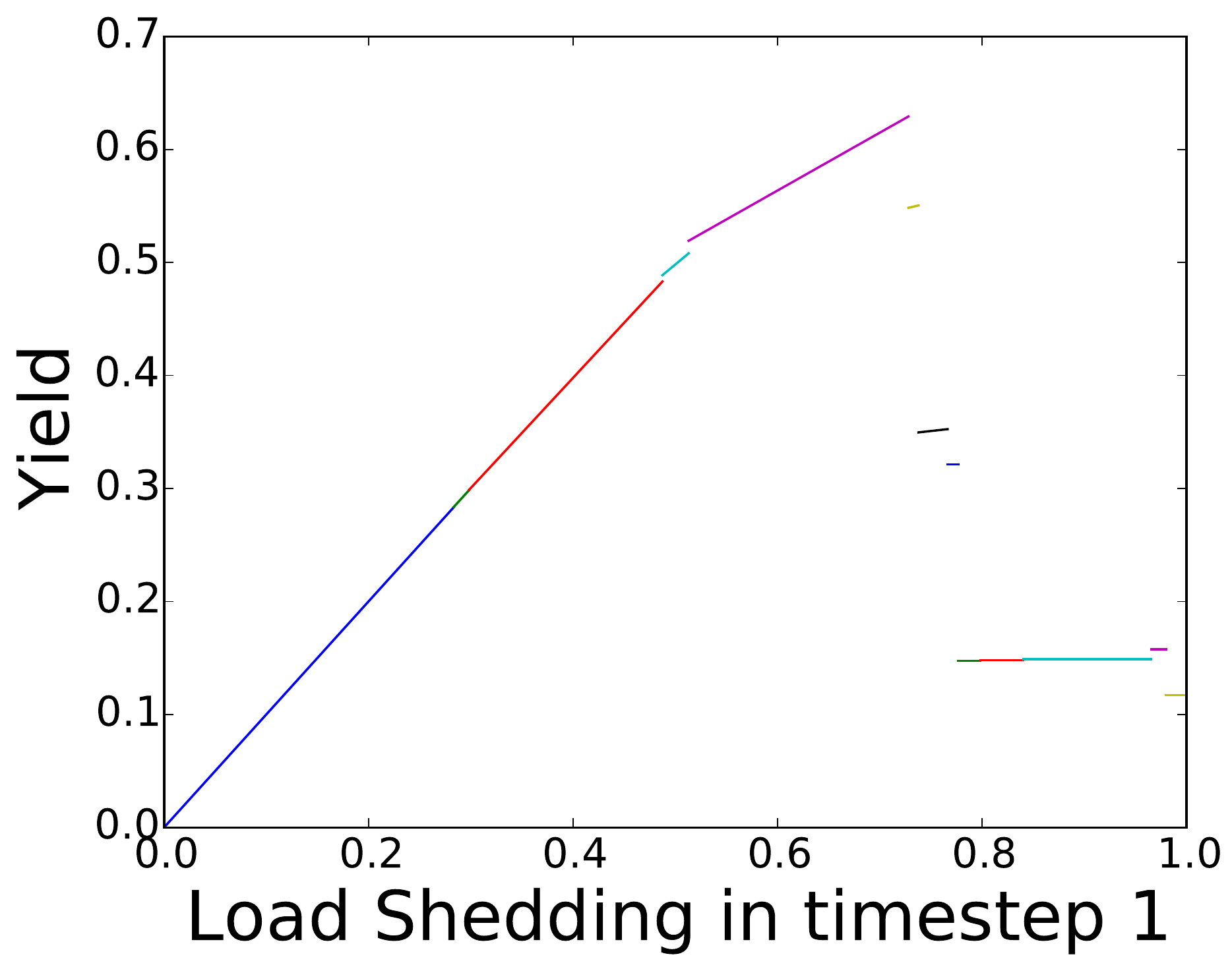}}
\caption{The yield when considering 120 realizations for: (a) the second
  time step and for $T=3$, given that an optimal load shedding scaling is done in the first
  time step, and (b) the first time step and for $T=2$. \label{fig:line2_2_vs_3}}
\end{figure}

In Fig.~\ref{fig:line2_5_steps} we set $T=5$. Fig.~\ref{fig:line2_5_steps}(a)
shows the yield as a function of load shedding in the first time steps and can
be used to find the optimal load shedding scaling for the first time
step. Fig.~\ref{fig:line2_5_steps}(b) shows the yield as a function of load
shedding in the second time step given that the optimal load shedding scaling
was done in the first time step. We can see that in this case making a small
amount of load shedding in the first two time steps results with a very good
final yield ($\approx 85\%$) compared to the case of $T=3$. We remark that
in the rest of the time steps the optimal control is to carry out very minimal
load shedding.

\begin{figure}
\centering
\subfigure[]{\includegraphics[width=0.485\columnwidth]{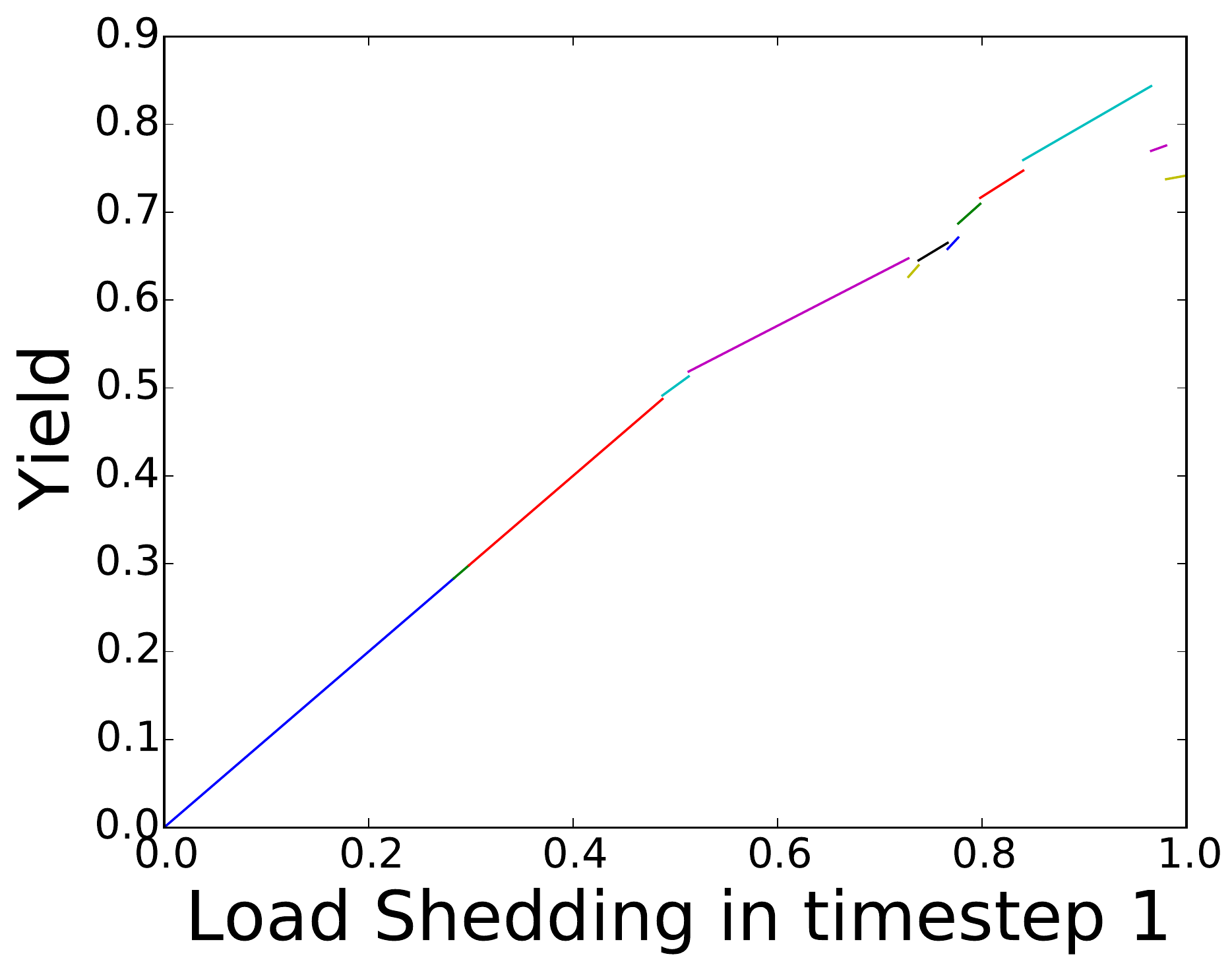}}
\subfigure[]{\includegraphics[width=0.485\columnwidth]{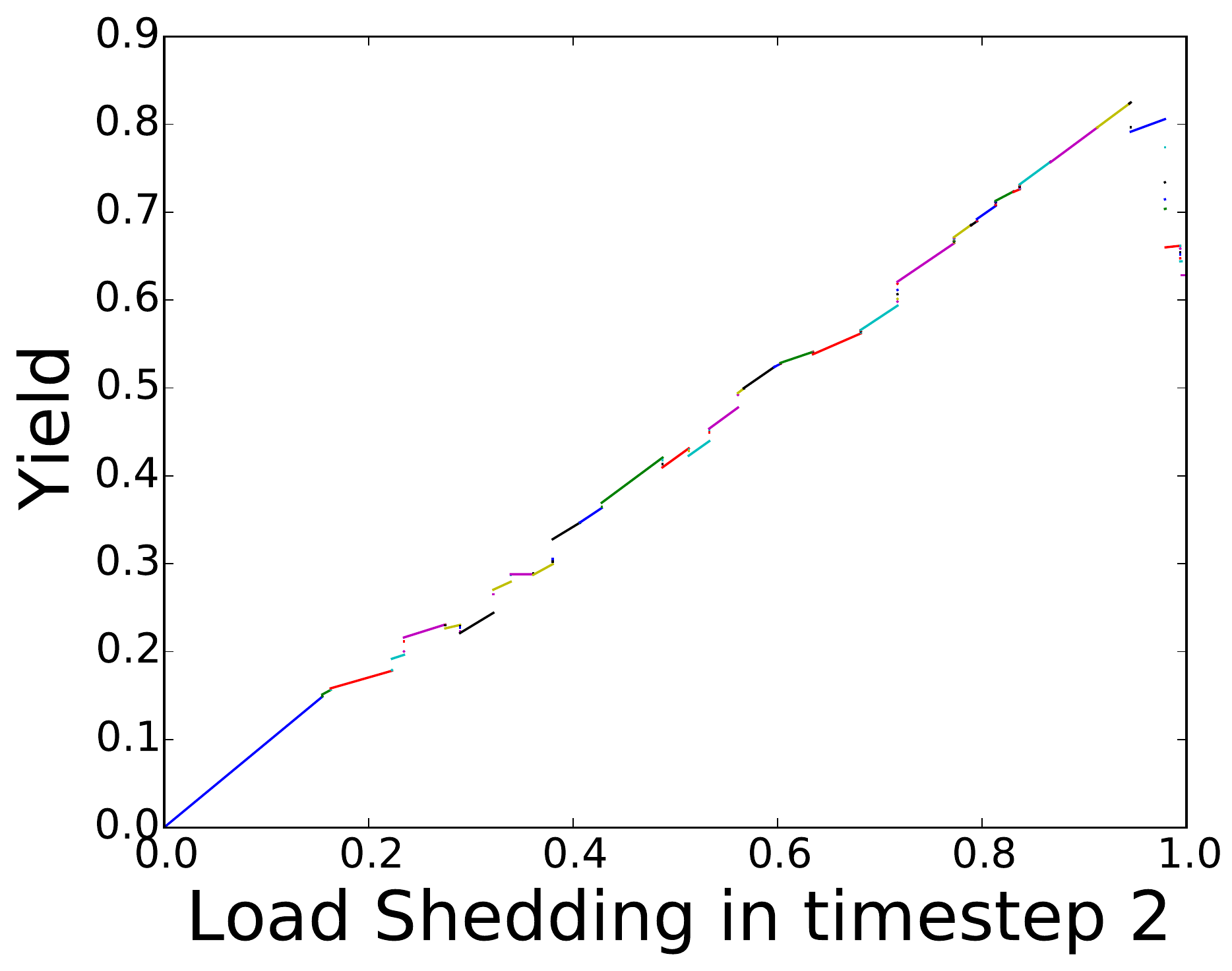}}
\caption{The yield when considering 120 realizations and $T=5$ for: (a) the
  first time step and for $T=2$, and (b) the second time step, given that 
  optimal load shedding is performed in the first time-step. \label{fig:line2_5_steps}}
\end{figure}

A natural question is what would the non-robust control algorithm from
\cite{bienstock2011optimal} would obtain in the considered model. To answer
this question, we compute the control using the algorithm from
\cite{bienstock2011optimal}. As expected, in the non-robust model, the
algorithm from \cite{bienstock2011optimal} is the best, since there is perfect
knowledge of the flows and the algorithm from \cite{bienstock2011optimal} is
optimal. However, when measurement errors are present, the algorithm from
\cite{bienstock2011optimal} performs much worse than the robust algorithm
presented here. Table~\ref{tab:yields} shows the yield obtained by both
algorithm. Additionally, Table~\ref{tab:yields} further demonstrates that
increasing $T$ improves the yield. We remark that the non-robust model refers
to the model from \cite{bienstock2011optimal} and the robust model refers
to the model considered in this paper.

\begin{table}
  \centering
  \begin{tabular}{|l|l|l|l|l|}
    \hline
    \multicolumn{1}{|c|}{$T$} & 2 & 3 & 4 & 5 \\ \hline
    Non-robust solution &  \multirow{2}{*}{$65.46\%$} &
    \multirow{2}{*}{$65.46\%$} & \multirow{2}{*}{$74.44\%$} & \multirow{2}{*}{$86.84\%$}\\ 
    and {\em non-robust model} & & & & \\ \hline
    Non-robust solution &  \multirow{2}{*}{$31.92\%$} &
    \multirow{2}{*}{$30.46\%$} & \multirow{2}{*}{$47.75\%$} & \multirow{2}{*}{$23.07\%$}\\ 
    and {\em robust model} & & & & \\ \hline
    robust solution &  \multirow{2}{*}{$62.19\%$} & \multirow{2}{*}{$62.19\%$} &
    \multirow{2}{*}{$70.73\%$} & \multirow{2}{*}{$78.36\%$}\\ 
    and {\em robust model} & & & & \\ \hline

  \end{tabular}
\vspace{5pt}
  \caption{Yields obtained from performing the solution obtained by the
    robust and non-robust algorithms under the robust and non-robust model.}
  \label{tab:yields}  
\end{table}

We also compute a $95\%$ confidence interval on the yield obtained when
performing the optimal load shedding. The confidence interval is computed for
different number of realizations and $T=5$ and it is shown in
Table~\ref{tab:confidence}. The obtained confidence interval are overall very
good. As expected, we can see improvement in the confidence interval as we
take more realizations.

\begin{table}
  \centering
  \begin{tabular}{|l|l|l|l|l|}
    \hline
    \multicolumn{1}{|c|}{$T$} & 2 & 5 \\ \hline
    5 realizations & [$62.39\%,63.19\%$] & [$83.02\%,84.87\%$] \\ \hline
    60 realizations & [$62.7\%,63.08\%$] & [$83.84\%,84.62\%$]\\ \hline
    120 realizations & [$62.79\%,63.04\%$] & [$84.05\%,84.64\%$]\\ \hline
  \end{tabular}
\vspace{5pt}
  \caption{A $95\%$ confidence interval for optimal expected yield for different number of realizations.}
  \label{tab:confidence}  
\end{table}

\bibliography{cdc15}
\bibliographystyle{IEEEtranS}

\end{document}